\documentclass[12pt,a4paper]{amsart}

\usepackage[utf8]{inputenc}
\usepackage[T1]{fontenc}
\usepackage[english]{babel}
\usepackage[hidelinks]{hyperref}
\usepackage{amsfonts}
\usepackage{amsthm}
\usepackage{mathtools}
\usepackage{amssymb}
\usepackage{cancel}
\usepackage[margin=3cm]{geometry}
\usepackage{emptypage}
\usepackage{graphicx}
\usepackage{multirow}

\theoremstyle{plain}
\newtheorem{Th}{Theorem}[section]
\newtheorem{Th*}{Theorem}
\newtheorem{Lemma}[Th]{Lemma}
\newtheorem{Cor}[Th]{Corollary}
\newtheorem{Prop}[Th]{Proposition}

\theoremstyle{definition}
\newtheorem{Def}[Th]{Definition}
\newtheorem{Conj*}{Conjecture}

\newtheorem{?}[Th]{Problem}
\newtheorem{Ex}[Th]{Example}
\newtheorem*{Not}{Notation}

\newcommand{\Supp}{\operatorname{Supp}}
\newcommand{\Bs}{\operatorname{Bs}}

\begin{document}
	
	\title{On Lang's conjecture for some product-quotient surfaces}
	
	\author{Julien Grivaux \and Juliana Restrepo Velasquez \and Erwan Rousseau}

	\address{Julien Grivaux \\ CNRS, I2M (Marseille) \& IH\'ES}
	\email{jgrivaux@math.cnrs.fr}
	
	\address{Juliana Restrepo Velasquez \\ Aix Marseille Univ\\ 
		CNRS, Centrale Marseille, I2M\\
		Marseille\\
		France} 
	\email{juliana.restrepo-velasquez@univ-amu.fr}
	
	\address{Erwan Rousseau \\ Institut Universitaire de France
	\& Aix Marseille Univ\\ 
		CNRS, Centrale Marseille, I2M\\
		Marseille\\
		France} 
	\email{erwan.rousseau@univ-amu.fr}
	
	\thanks{The third author is partially supported by the ANR project \lq\lq FOLIAGE\rq\rq{}, ANR-16-CE40-0008.}
	
	
	\begin{abstract}
	We prove effective versions of algebraic and analytic Lang's  conjectures for product-quotient surfaces of general type with $P_g=0$ and $c_1^2=c_2$.

	\end{abstract}
	
	\maketitle
	\tableofcontents
	
	\section{Introduction} 
	Lang's conjecture asserts that curves of fixed geometric genus on a surface of general type form a bounded family.
An effective version of this conjecture can be stated in the following way:

\begin{Conj*} 
	(Lang-Vojta). Let $S$ a smooth projective surface of general type. Then there exist real numbers $A$, $B$, and a strict subvariety $Z\subset S$ such that, for any holomorphic map $f:C\rightarrow S$ satisfying $f(C)\nsubseteq Z$, where $C$ is a smooth projective curve,
	\[
		\deg f(C)\leq A(2g(C)-2)+B.
	\]
\end{Conj*}

Bogomolov proved this conjecture for surfaces of general type satisfying $c_1^2-c_2>0$  \cite{Bogomolov}. He actually proved that such surfaces have big cotangent bundle, and that the conjecture follows from this fact. Unfortunately, this approach does not provide effective information about $A$ and $B$. However, effective results for such surfaces have been obtained more recently by Miyaoka \cite{miyaoka}. \\

On the other hand, the analytic version of Lang's conjecture is stated as follows:

\begin{Conj*} 
	(Green-Griffiths-Lang). Let $S$ a smooth projective surface of general type. Then there exists a strict subvariety $Z\subset S$ such that for any non constant holomorphic map $f:\mathbb{C}\rightarrow S$,
	\[
		f(\mathbb{C})\subset Z.
	\]
\end{Conj*}

Bogomolov's result has been generalized to the analytic case by McQuillan in his proof of this conjecture for surfaces of general type with $c_1^2-c_2>0$ \cite{mcquillan}.\\

Here, we are interested in \emph{product-quotient surfaces}, i.e., in the minimal resolutions of quotients $X:=(C_1\times C_2)/G$, where $C_1$ and $C_2$ are two smooth projective curves of respective genera $g(C_1),g(C_2)\geq 2$, and $G$ is a finite group, acting faithfully on each of them and diagonally on the product. These surfaces generalize the so-called Beauville surfaces (the particular case where the group action is free). Thanks to the work that I. Bauer, F. Catanese, F. Grunewald and R. Pignatelli, started and carried through in \cite{BC}, \cite{BCG} and \cite{BCGP}, we finally have a complete classification of product-quotient surfaces of general type with geometric genus $P_g=0$ in \cite{BP}.
We want to study product-quotient surfaces of general type with geometric genus $P_g=0$ such that $c_1^2-c_2=0$.  Note that $P_g=0$ implies $c_1^2+c_2=12$, then the condition $c_1^2-c_2=0$ is equivalent to $c_1^2=6$. These surfaces are a limit case not covered by Bogomolov's theorem; however, they satisfy the criterion given in (\cite{RR}, Theorem $1$), which ensures the bigness of their cotangent bundle. 

In this paper, we prove Conjectures 1 and 2, when $S$ is a product-quotient surface of general type with geometric genus $P_g=0$ and $c_1^2=6$. The key point in our approach is the fact that $c_1^2=6$ implies the bigness of the line bundle $\mathcal{O}(K_S-E)$, where $K_S$ is the canonical divisor and $E$ the exceptional divisor on $S$.

First, we give an alternative proof for the bigness of $\Omega_S$, producing explicit symmetric tensors on $S$ coming from $\mathcal{O}(K_S-E)$, which allows us to control rational curves on it. More precisely we prove the following theorem:

\begin{Th*}\label{rational curves}
	Let $S$ be a product-quotient surface of general type such that $P_g(S)=0$. If $c_1(S)^2=6$, then:
	\begin{enumerate}
	
		\item The line bundle bundle $K_S-E$ and the cotangent bundle $\Omega_S$ are big. In particular, $S$ contains only a finite number of rational and elliptic curves.
		\item For any non constant holomorphic map $f:\mathbb{P}^1\rightarrow S$, 
		\[
			f(\mathbb{P}^1)\subset E\cup \mathbb{B}(K_S-E),
		\]
		where $E$ is the exceptional divisor on the resolution $S$ and 
		\[
			\mathbb{B}(K_S-E):=\displaystyle\bigcap_{m>0}\Bs(m(K_S-E))
		\]
		is the stable base locus of $K_S-E$.
	\end{enumerate}
\end{Th*}
	
On the other hand, we prove the following result for all product-quotient surfaces:
	
\begin{Th*}\label{effective}
	Let $S$ be a product-quotient surface. If $f:C\rightarrow S$ is a holomorphic map such that $f(C)\nsubseteq E$, where $C$ is a smooth projective curve and $E$ is the exceptional divisor on $S$, then
	\[
		\deg f^*(K_S-E)\leq 2(2g(C)-2).
	\]
\end{Th*}

Note that in Conjecture 1, one can take $\deg f(C)=\deg f^*\mathcal{L}$ for $\mathcal{L}$ a positive (ample or big) line bundle on $S$. Then, taking $\mathcal{L}=\mathcal{O}(K_S-E)$, the previous result provides an effective proof of Conjecture 1 for our particular case.

Finally, our approach also lets us control elliptic and more generally, entire curves on $S$. More precisely, we prove Conjecture 2:

\begin{Th*}\label{entire curves}
	Let $S$ be a product-quotient surface of general type such that $P_g(S)=0$. If $c_1(S)^2=6$, then for any non constant holomorphic map $f:\mathbb{C}\rightarrow S$,
	\[
		f(\mathbb{C})\subset E\cup \mathbb{B^+}(K_S-E),
	\]
	where $E$ is the exceptional divisor and  
	\[
		\mathbb{B^+}(K_S-E):=\displaystyle\bigcap_{m>0}\Bs(m(K_S-E)-A)
	\]
	with $A$ an ample line bundle, is the augmented base locus of $K_S-E$.
\end{Th*}
	
	\section{Preliminaries}
	In this section we are going to recall some definitions and results that will be used throughout this paper. 
	
	\subsection{Product-quotient surfaces}
	\begin{Def}
	A \emph{product-quotient surface} $S$ is the minimal resolution of the singularities of a quotient $X:=(C_1\times C_2)/G$, where $C_1$ and $C_2$ are two smooth projective of respective genera $g(C_1),g(C_2) \geq 2$, and $G$ is a finite group, acting faithfully on each of them and diagonally on the product. The surface $X:=(C_1\times C_2)/G$ is called \emph{the quotient model} of $S$ \cite{BP}.
\end{Def}

Let $S$ be a product-quotient surface. Let $\varphi:S\rightarrow X$ be the resolution morphism of the singularities of $X:=(C_1\times C_2)/G$, and let $p_1:X\rightarrow C_1/G$ and $p_2:X\rightarrow C_2/G$ be the two natural projections. Let us define $\sigma_1:S\rightarrow C_1/G$ and $\sigma_2:S\rightarrow C_2/G$ to be the compositions $p_1\circ \varphi$ and $p_2\circ \varphi$ respectively. Thus, we have the following commutative diagram encoding all this information:

\begin{center}
	\raisebox{-0.5\height}{\includegraphics[scale=1.25]{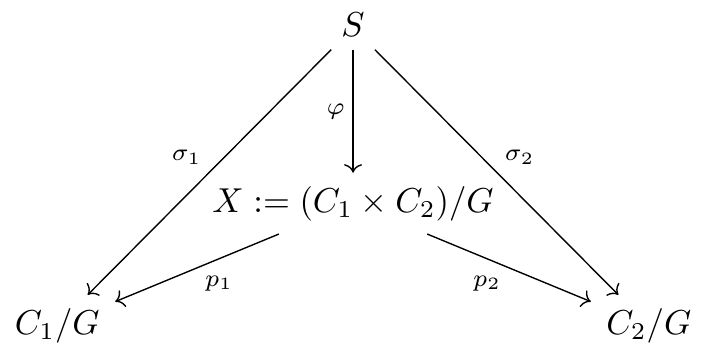}}
\end{center}

The surface $X:=(C_1\times C_2)/G$ has a finite number of singularities, since there are finitely many points on $C_1\times C_2$ with non trivial stabilizer. Moreover, since $G$ is finite, the stabilizers are cyclic groups (\cite{Farkas}, III $7.7$) and so, the singularities of $X$ are \emph{cyclic quotient singularities}. Thus, if $(x,y)\in C_1\times C_2$ with non trivial stabilizer $H_{(x,y)}$, then, around the singularity $\overline{(x,y)}\in X:=(C_1\times C_2)/G$, $X$ is analytically isomorphic to the quotient $\mathbb{C}^2/ \mathbb{Z}_n$, where $n=|H_{(x,y)}|$ and the action of the cyclic group $\mathbb{Z}_n = \langle \xi\rangle$ on $\mathbb{C}_2$ is defined by $\xi(z_1,z_2)=(\xi z_1,\xi^a z_2)$, being $n$ and $a$ coprime integers such that $1\leq a\leq n-1$ and $\xi$ is a primitive $n$-th root of unity. In this case, the cyclic quotient singularity is called  \emph{singularity of type} $\frac{1}{n}(1,a)$.\\
Note that singular points of type $\frac{1}{n}(1,a)$, are also of type $\frac{1}{n}(1,a')$ where $a'$ is the multiplicative inverse of $a$ in $\mathbb{Z}_n^*$; namely $a'$ is the unique integer $1\leq a'\leq n-1$ such that $aa'\equiv 1$ modulo $n$ (\cite{BPHV}, III).

The exceptional fiber of a cyclic quotient singularity of $X$ of type $\frac{1}{n}(1,a)$ on the the minimal resolution $S$, is a \emph{Hizerbruch-Jung string} (H-J string), that is to say, a connected union $L=\sum_{i=0}^lZ_i$ of smooth rationals curves $Z_1,\dots, Z_l$ with self-intersection numbers less or equal than $-2$, and ordered linearly so that $Z_iZ_{i+1}=1$ for all $i$ and $Z_iZ_j=0$ if $\vert i-j\vert\geq 2$ (\cite{BPHV}, III $5.4$), \cite{Fujiki}. Then, the exceptional divisor $E$ on the minimal resolution $S$ is the connected union of disjoint H-J strings each of them being the fiber of each singularity of $X:=(C_1\times C_2)/G$.

The self-intersection numbers $Z_i^2=-b_i$ are given by the formula 
\[
	\dfrac{n}{a}=b_1-\dfrac{1}{b_2-\dfrac{1}{\cdots-\dfrac{1}{b_l}}}
\]
Abusing slightly of the notation, we denote the right part of the formula by $[b_1,\cdots,b_l]$. 

Moreover, 
\[
	\dfrac{n}{a}=[b_1,\cdots, b_l] \text{ if and only if }\dfrac{n}{a'}=[b_l,\cdots, b_1]
\]
Cyclic quotient singularities of type $\frac{1}{n}(1,n-1)$ are particular cases of \emph{rational double points}: all the curves $Z_i$ have self-intersection equal to $-2$. Singularities of type $\frac{1}{2}(1,1)$ are called \emph{ordinary double points}. 
	
On the other hand, Serrano's paper (\cite{Serrano}, Proposition $2.2$) tells us that the irregularity of $S$, defined by $q(S):=h^1(S,\mathcal{O}_S)$, is given by the formula:
\[
	q(S)=g(C_1/G)+g(C_2/G).
\]
Now, if $S$ is of general type, then $q(S)\leq P_g(S)$. Therefore, we have that $S$ is a product-quotient surface of general type with $P_g=0$ if and only if $\chi(\mathcal{O}_S)=1$ and $C_1/G\cong C_2/G\cong \mathbb{P}^1$. Moreover, using the Noether's formula we see that the condition $P_g=0$ also implies that $c_1^2+c_2=12$. 

The classification of product-quotient surfaces with $P_g=0$ was started by I. Bauer and F. Catanese in \cite{BC}; they classified the surfaces $X=(C_1\times C_2)/G$ with $G$ being an abelian group acting freely and $P_g(X)=0$. Later in \cite{BCG}, both of them and F. Grunewald, extended this classification to the case of an arbitrary group $G$. Not long after, R. Pignatelli jointed them, and in \cite{BCGP} they dropped the assumption that $G$ acts freely on $C_1\times C_2$; they classified product-quotient surfaces with $P_g=0$ whose quotient model $X$ has at most canonical singularities. 

Finally in \cite{BP}, I. Bauer and R. Pignatelli dropped any restiction on the singularities of $X$ and gave a complete classification of product-quotient surfaces  $S$ of general type with $P_g=0$ and $K_S^2>0$. Moreover, they proved that there are exactly $73$ irreducible families of surfaces of this kind, and all but one of them, are in fact minimal surfaces; more precisely they proved the following result:

\begin{Th}[\cite{BP}, Theorem 0.3]\label{BP}
	\
	\begin{enumerate}
		\item Minimal product-quotient surfaces with $P_g=0$ form exactly $72$ irreducible families.
		\item There is exactly one product-quotient surface with $K_S^2>0$ which is non minimal. It has $K_S^2=1$, $\pi_1(S)=\mathbb{Z}_6$.
	\end{enumerate}
	
\end{Th}  

The irreducible families mentioned in the first part of this theorem, are listed in tables $1$ and $2$ in \cite{BP}.
	
	\subsection{Isotrivial fibrations}
	\begin{Def} 
	A \emph{fibration} is a morphism from a smooth projective surface onto a smooth curve, with connected fibers. A fibration is called \emph{isotrivial fibration}, if all its smooth fibers are mutually isomorphic. A surface is called \emph{isotrivial surface} if it admits an isotrivial fibration.
\end{Def}

A product-quotient surface $S$ is an example of an isotrivial surface: it admits two natural isotrivial fibrations $\sigma_1:S\rightarrow C_1/G$ and $\sigma_2:S\rightarrow C_2/G$ whose smooth fibers are all isomorphic to $C_2$ and $C_1$ respectively. 

\begin{Def}
	A smooth projective surface is called \emph{standard isotrivial surface} if it is isomorphic to a product-quotient surface.
\end{Def}

In Serrano's paper \cite{Serrano} it is proved that any isotrivial surface is birationally equivalent to a standard one, more precisely, if $\gamma:Z\rightarrow C$ is isotrivial, then there exist a quotient $(C_1\times C_2)/G$ where $C_1$ is isomorphic to the general fiber of $\gamma$ and $G$ is a finite group, acting faithfully on $C_1$ and $C_2$ and diagonally on the product; such that $Z$ is birational to $(C_1\times C_2)/G$, $C$ is isomorphic to $C_2/G$, and the following diagram commutes:

\begin{center}
	\raisebox{-0.5\height}{\includegraphics[scale=1.25]{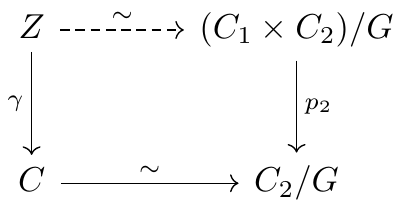}}
\end{center}

We also find in Serrano's paper a description of the singular fibers that can arise in a standard isotrivial surface, i.e., the possible singular fibers of its natural fibrations. Namely:

\begin{Th}[\cite{Serrano}, Theorem $2.1$]
	Let $S$ a standard isotrivial surface and let consider the fibration $\sigma_2:S\rightarrow C_2/G$. Let $y\in C_2$ and $H_y$ its stabilizer. If $F$ is the fiber of $\sigma_2$ over $\overline{y}\in C_2/G$, then:
	\begin{enumerate}
		\item The reduced structure of $F$ is the union of an irreducible smooth curve $Y$, called $\textit{the central component}$ of $F$, and either none or at least two mutually disjoint $\textit{H-J strings}$, each one meeting $Y$ at one point. These strings are in one-to-one correspondence with the branch points of $C_1\rightarrow C_1/H_y$.
		\item The central component $Y$ is isomorphic to $C_1/H_y$ and it has multiplicity equal to $\vert H_y\vert$ in $F$. The intersection of a string with $Y$ is transversal, and it takes place at only one of the end components of the string.
		\item If $L=\sum_{i=1}^{n}Z_i$ is a H-J string on $F$ and $Y'$ is the central component of the fiber of $\sigma_1:S\rightarrow C_1/G$ over $\sigma_1(L)$, then $L$ meets $Y'$ and $Y$ at opposite ends, i.e., either $Z_1Y=Z_nY'=1$ or $Z_nY=Z_1Y'=1$.
	\end{enumerate}
\end{Th}

If $F$ contains exactly $r$ H-J strings $L_1,\cdots,L_r$, where each $L_i$ is the resolution of a cyclic quotient singularity of type $\frac{1}{n_i}(1,a_i)$, then we know that the central component $Y$ satisfies that
\[
	Y^2=-\sum_{i=1}^r\dfrac{a_i}{n_i}
\] 
(\cite{Polizzi}, Proposition 2.8). All these properties hold for any fiber of $\sigma_1$.\\

Finally, Serrano's paper also provides an expression for the canonical bundle of a standard isotrivial surface in terms of the fibers of the two natural fibrations. Namely:

\begin{Th}[\cite{Serrano}, Theorem 4.1]
	Let $S$ be a standard isotrivial surface with associated fibrations $\sigma_1:S\rightarrow C_1/G$ and $\sigma_2:S\rightarrow C_2/G$. Let $\{n_iN_i\}_{i\in I}$ and $\{m_jM_j\}_{j\in J}$ denote the components of all singular fibers of $\sigma_1$ and $\sigma_2$ respectively, with their multiplicities attached. Finally, let $\{Z_t\}_{t\in T}$ be the set of curves contracted to points by $\sigma_1\times \sigma_2$, i.e, the exceptional locus on $S$. Then we have
	\[
		K_S=\sigma_1^*(K_{C_1/G})+\sigma_2^*(K_{C_2/G})+\sum_{i\in I}(n_i-1)N_i+\sum_{j\in J}(m_i-1)M_i+\sum_{t\in T}Z_t.
	\]
\end{Th}

The fibrations $\sigma_1:S\rightarrow C_1/G$ and $\sigma_2:S\rightarrow C_2/G$ can be thought as foliations $\mathcal{F}_1$ and $\mathcal{F}_2$ on $S$, such that Serrano's formula can be written as follows:	
\[
	K_S=\mathcal{N}^*_{\mathcal{F}_1}\otimes \mathcal{N}^*_{\mathcal{F}_2} \otimes \mathcal{O}_S(E)
\]
where $\mathcal{N}^*_{\mathcal{F}_1}$ and $\mathcal{N}^*_{\mathcal{F}_2}$ are the respective conormal line bundles, and $E$ is the exceptional divisor on $S$ (see \cite{Brunella}, p.30). 
	
	\section{Product-quotient surfaces with $P_g=0$ and $c_1^2=c_2$}
	In this section, we are going to study product-quotient surfaces of general type with geometric genus $P_g=0$ and $c_1^2-c_2=0$. Recall that $P_g=0$ implies $c_1^2+c_2=12$, then the last condition is equivalent to having $c_1^2=6$.

I. Bauer and R. Pignatelli give us information in this case. Namely, we know that this kind of surfaces form exactly $8$ irreducible families and we also have a complete description of their quotient models. In the following table we summarize some information that might be useful. 
\vspace{0.25cm}
\renewcommand{\arraystretch}{1.5}
\begin{center}
	\begin{tabular}{|c|c|c|c|c|c|c|}
		\hline
		$c_1^2(S)$ & Singularities of $X$ & $G$ & $|G|$ & $g(C_1)$ & $g(C_2)$ & \parbox{2cm}{\vspace{0.2cm}Number of irreducible families\vspace{0.2cm}}\\
		\hline
		\multirow{6}{*}{$6$}& \multirow{6}{*}{Two of type $\frac{1}{2}(1,1)$} & $\mathbb{Z}_2\times D_4$ & $16$ & $3$ & $7$ & $1$ \\ 
		\cline{3-7}
		&  & $\mathbb{Z}_2\times \mathfrak{S}_4$ & $48$ & $19$ & $3$ & $1$ \\ 
		\cline{3-7} 
		& & $\mathfrak{A}_5$ & $60$ & $4$ & $16$ & $1$ \\ 
		\cline{3-7} 
		&  & $\mathbb{Z}_2\times \mathfrak{S}_5$ & $240$ & $19$ & $11$ & $1$ \\ 
		\cline{3-7} 
		&  & $\text{PSL}(2,7)$ & $168$ & $19$ & $8$ & $2$ \\ 
		\cline{3-7}
		&  & $\mathfrak{A}_6$ & $360$ & $19$ & $16$ & $2$ \\ 
		\hline 
	\end{tabular}
\end{center}
\vspace{0.2cm}
For even more information see table 1 in \cite{BP}.\\

Let $S$ be a product-quotient surface and let us suppose that $S$ is of general type with $P_g(S)=0$ and $c_1^2=6$. We recall the following commutative diagram: 

\begin{center}
	\raisebox{-0.5\height}{\includegraphics[scale=1.25]{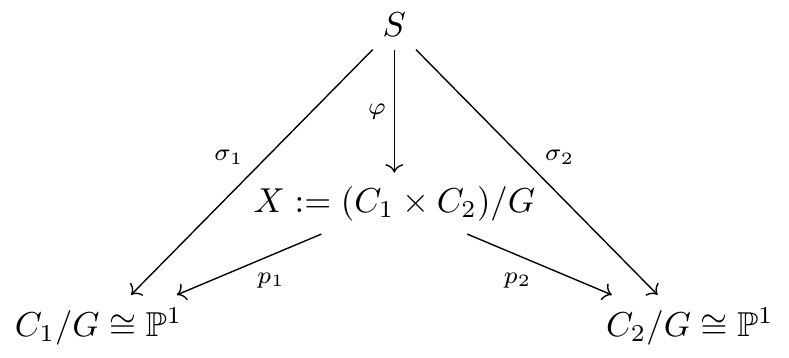}}
\end{center}

We know that the quotient model $X$ of $S$ has only two cyclic quotient singularities of type $\frac{1}{2}(1,1)$. Since these singularities are canonical, $K_S$ is nef and then $S$ is minimal. Theorem \ref{BP}, proved in (\cite{BP}, Theorem $0.3$), gives us another argument for the minimality of $S$.

Around one of the two singular points, $X$ is analytically isomorphic to the quotient $\mathbb{C}^2/ \mathbb{Z}_2$, where the cyclic group $\mathbb{Z}_2$ acts on $\mathbb{C}^2$ by $(z_1,z_2)\rightarrow(-z_1, -z_2)$. This quotient is an affine subvariety of $\mathbb{C}^3$, with coordinates $u=z_1^2, v=z_1z_2, w=z_2^2$, defined by the equation $uw=v^2$ (\cite{Reid}, Prop-Def 1.1, Example 1.2). Moreover, if $\mu_1,\mu_2$ are local coordinates on $S$, the resolution morphism $\varphi$ is locally given by 
\[
	\varphi(\mu_1,\mu_2)=(u=\mu_1,v=\mu_1\mu_2,w=\mu_1 \mu_2^2)
\] 
(\cite{Reid}, Example 3.1).  Therefore, we have the following relations between the local coordinates $z_1, z_2$ and $\mu_1, \mu_2$:
\begin{equation*}
	\left\lbrace
	\begin{array}{l}
		z_1 = \mu_1^{1/2}\\
		z_2 = \mu_1^{1/2}\mu_2
	\end{array}
	\right.
\end{equation*}

On the other hand, the exceptional fiber of a cyclic quotient singularity $\frac{1}{2}(1,1)$ on the minimal resolution $S$, is a H-J string formed by only one smooth rational curve with self-intersection number equal to $-2$. Using the local coordinates $\mu_1,\mu_2$ on $S$ we see that it is given by the set of points $(\mu_1,\mu_2)$ such that $\mu_1=0$.
	
We denote by $E$ the exceptional divisor on the minimal resolution $S$ of $X$. Since $X$ has only two cyclic quotient singularities of type $\frac{1}{2}(1,1)$, then $E$ is the disjoint union of two rational curves with self-intersection number equal to $-2$. Moreover, $E$ is locally defined by the equation $\mu_1=0$.

\begin{Not}
For the rest of this section (subsections 3.1, 3.2, 3.3), we will denote by $S$ a product-quotient surface of general type such that $P_g(S)=0$ and $c_1(S)^2=6$, by $X:=(C_1\times C_2)/G$ its quotient model and by $E$ the exceptional divisor.
\end{Not}
		
	\subsection{Bigness of the cotangent bundle}
	We denote by $\Lambda$ the set of points of $C_1\times C_2$ with non trivial stabilizer. Recall that $\Lambda$ is a finite set.

Let us first describe a natural way to produce sections of $\mathcal{S}^{2m}\Omega_X$, from sections of $K_S^{\otimes m}$ using the following diagram:

\begin{center}
	\raisebox{-0.5\height}{\includegraphics[scale=1.25]{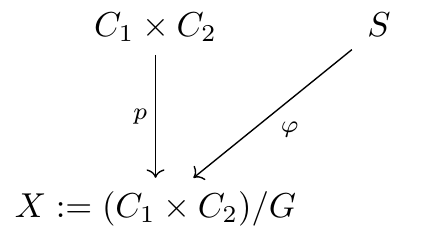}}
\end{center}

Let $\omega$ be a section of $\mathcal{K}_S^{{\otimes}m}$. The pushforward $\varphi_*$ of $\omega$ is a section of $\mathcal{K}_X^{{\otimes}m}$ defined on, that can be lifted by the pullback $p^*$ to a section of $(\mathcal{K}_{C_1\times C_2}^{{\otimes}m})^G$ defined outside of $\Lambda$. However, since $\text{codim}_{C_1\times C_2}\Lambda=2$, this section uniquely extends to a section defined on $C_1\times C_2$ . Moreover, the canonical isomorphism between $\mathcal{K}_{C_1\times C_2}$ and $\lambda_1^*\Omega_{C_1}\otimes \lambda_2^*\Omega_{C_2}$ where $\lambda_1:C_1\otimes C_2\rightarrow C_1$ and $\lambda_2:C_1\otimes C_2\rightarrow C_2$ are the projections, allows us to identify the sections of $(\mathcal{K}_{C_1\times C_2}^{{\otimes}m})^G$ with sections of $(\mathcal{S}^{2m}\Omega_{C_1\times C_2})^G$. Therefore, we get a section of $(\mathcal{S}^{2m}\Omega_{C_1\times C_2})^G$ which descend by $p_*$ to a section of $\mathcal{S}^{2m}\Omega_X$ defined on  the regular part of $X$. We denote this section by $\Theta(\omega)$. 

Let us denote by $\Gamma(\omega)$ the pullback of $\Theta(\omega)$ by $\varphi^*$. Note that $\Gamma(\omega)$ is, a priori, a section of $\mathcal{S}^{2m}\Omega_S$ defined outside of the exceptional divisor $E$.\\

If we start with global sections of $\mathcal{K}_S^{\otimes m}$, this process is summarized in the following commutative diagram:
\begin{center}
	\raisebox{-0.5\height}{\includegraphics[width=\textwidth]{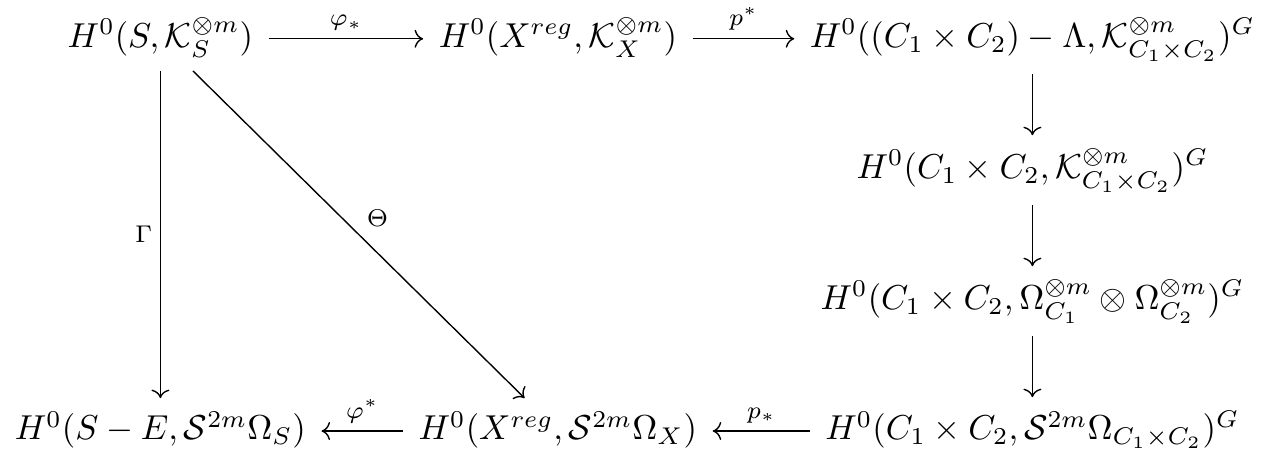}}
\end{center}

The following proposition ensures that taking global sections of $\mathcal{K}_S^{\otimes m}
$ vanishing along $E$, at least with multiplicity $m$, is a sufficient condition to obtain global sections of $\mathcal{S}^{2m}\Omega_S$.

\begin{Prop} \label{1}
	If $\omega$ is a global section of $\mathcal{O}(m(K_S-E))$, then $\Gamma(\omega)$ naturally extends to a well-defined global section of $\mathcal{S}^{2m}\Omega_S$.
\end{Prop}

\begin{proof}  
	Let $\omega\in H^0(S,\mathcal{O}(m(K_S-E)))$. Following the previous diagram, we get $\Theta(\omega)\in H^0(X^{reg},\mathcal{S}^{2m}\Omega_X)$. By definition, the corresponding section on $C_1\times C_2$ can be written locally, let us say around a fixed point, as 
	\[
		a(z_1,z_2){dz_1}^{m}{dz_2}^{m}.
	\]
	
	Using the change of coordinates $z_1=\mu_1^{1/2}$ and $z_2=\mu_1^{1/2}\mu_2$ given by $\varphi$ at singular points of $X$, we get that the pullback by $\varphi^*$ of $\Theta(\omega)$, which is nothing else than $\Gamma(\omega)$, can be written locally as 
	\[
		\sum_{j=0}^m\binom{m}{j}\frac{\mu_2^{m-j}(a\circ\varphi)(\mu_1,\mu_2)}{2^{2m-j}\mu_1^{m-j}}d\mu_1^{2m-j}d\mu_2^j
	\]
	and it naturally extends to a well-defined global section of $\mathcal{S}^{2m}\Omega_S$, since $a\circ\varphi$ vanishes along $E$ at least with multiplicity $m$.
\end{proof}

\begin{Prop} \label{2}
	The line bundle $\mathcal{O}(K_S-E)$ is big.
\end{Prop}

\begin{proof}
	Since $S$ is a minimal surface of general type, then the canonical divisor $K_S$ is nef. Therefore, by the asymptotic Riemann-Roch theorem, we have that 
	\[
		h^0(S,\mathcal{K}_S^{\otimes m})=\frac{m^2c_1(S)^2}{2}+O(m),
	\]
	but $c_1(S)^2=6$, so
	\[
		h^0(S,\mathcal{K}_S^{\otimes m})=3m^2+O(m).
	\]
	Thus, there exists a positive real number $M$ such that 
	\[
		3m^2-Mm \leq h^0(S,\mathcal{K}_S^{\otimes m})
	\]
	for $m$ large enough.
	
	On the other hand, let $\omega$ be a section of $\mathcal{K}_S^{\otimes m}$.  The corresponding section on $C_1\times C_2$ can be written locally, around a fixed point, as 
	\[
		a(z_1,z_2)({dz_1}\wedge{dz_2})^{m}
	\]
	where $a$ is a holomorphic function defined as
	\[
		a(z_1,z_2)=\sum_{i,j}a_{ij}z_1^iz_2^j
	\]
	
	Using the change of coordinates $z_1=\mu_1^{1/2}$ and $z_2=\mu_1^{1/2}\mu_2$ given by $\varphi$ at singular points of $X$, we see that $\omega$ vanishes along $E$, at least with multiplicity $m$ if $a_{i,j}=0$, for every $i,j$ such that $i+j < 2m$, and this gives us, $1+2+\cdots +2m$ sufficient conditions. However, the section is invariant by the action of $G$, then $a_{ij}=0$ for all $i,j$ such that $i+j$ is odd since around a singular point, $a(z_1,z_2)$ is invariant by the action of its stabilizer $H\simeq \mathbb{Z}_2$. Therefore we just need to consider half of the conditions. Finally, since these conditions are given around one singular point, we just need to multiply by the number of singularities.
		
	Thus,
	\begin{align*}
		h^0(S,\mathcal{O}(m(K_S-E)))& \geq  h^0(S,\mathcal{K}_S^{\otimes m})- \frac{2(1+2+\cdots +2m)}{2}\\
		& \geq (3m^2-Mm)-(2m^2+m)\\
		&= m^2-(M+1)m
	\end{align*}
	for $m$ large enough.\\
	But for any $0<C<1$, we have that $m^2-(M+1)m\geq Cm^2$ for $m$ large enough.
	
	Therefore,
	\[
		h^0(S,\mathcal{O}(m(K_S-E)))\geq Cm^2
	\]
	for $m$ large enough, which means that $\mathcal{O}(K_S-E)$ is big (\cite{L}, Lemma $2.2.3$).
\end{proof}

\begin{Prop} \label{3}
	The cotangent bundle $\Omega_S$ is big.
\end{Prop}

\begin{proof}
	For the sake of simplicity, we are going to use the same notation to refer to a divisor and its associated line bundle, except when we explicitly denote it.
	
	In order to prove that $\Omega_S$ is big, we show that the line bundle $\mathcal{O}_{\mathbb{P}(T_S)}(1)$ is big, which is equivalent to see that $\mathcal{O}_{\mathbb{P}(T_S)}(k)$, as a divisor, is the sum of an ample divisor and an effective divisor for a $k$ large enough (\cite{L}, Corollary $2.2.7$).
	
	Proposition \ref{2} tells us that $\mathcal{O}(K_S-E)$ is big, then there exists an ample line bundle $A$ and a positive integer $m$ such that 
	\[
		H^0(S,\mathcal{O}(m(K_S-E))\otimes A^{-1})\neq 0
	\] 
	and hence
	\[
		H^0(S,\mathcal{S}^{2m}\Omega_S\otimes A^{-1}) \neq 0.
	\]	
	However, $\mathcal{S}^{2m}\Omega_S \simeq \pi_*\mathcal{O}_{\mathbb{P}(T_S)}(2m)$ where $\pi:\mathbb{P}(T_S)\rightarrow S$ is the projective bundle associated to the tangent bundle $T_S$, and so we obtain that 
	\[
		H^0(\mathbb{P}(\Omega_S),\mathcal{O}_{\mathbb{P}(T_S)}(2m)\otimes \pi^{*}A^{-1}) \neq 0.
	\]
	But thinking of $\mathcal{O}_{\mathbb{P}(T_S)}(1)$ and $\pi^*A$ as divisors on $\mathbb{P}(T_S)$ and $S$ respectively, this last expression means that $\mathcal{O}_{\mathbb{P}(T_S)}(2m)- \pi^{*}A$ is an effective divisor.\\
	
	On the other hand, using an ampleness property on the projective bundle (\cite{L}, Proposition 1.2.7), we know that there exists a large enough positive integer $l$ such that $\mathcal{O}_{\mathbb{P}(T_S)}(1)+ l\pi^{*}A$ is an ample divisor on $\mathbb{P}(T_S)$.\\
	
	Finally, taking $k=2ml+1$ we get 
	\[
		\mathcal{O}_{\mathbb{P}(T_S)}(2ml+1) = \underbrace{ l(\mathcal{O}_{\mathbb{P}(T_S)}(2m)- \pi^{*}A)}_{\text{effective divisor}}+\underbrace{\mathcal{O}_{\mathbb{P}(T_S)}(1)+ l\pi^{*}A}_{\text{ample divisor}}
	\]
	and so $\mathcal{O}_{\mathbb{P}(T_S)}(1)$ is big.
\end{proof}
	
	\subsection{Rational curves}
	We already know that $\Omega_S$ is big and then, by Bogomolov's argument, there is only a finite number of rational curves on $S$; now we want to get more constraints. Recall that $S$ admits two natural isotrivial fibrations $\sigma_1:S\rightarrow C_1/G$ and $\sigma_2:S\rightarrow C_2/G$ and that they can be thought as foliations $\mathcal{F}_1$ and $\mathcal{F}_2$ on $S$.

\begin{Lemma}\label{cc}
	The central component $Y$ of any singular fiber on $S$ is not rational. 
\end{Lemma}

\begin{proof}
	Recall that the only singularities of $X$ are two ordinary double points, i.e., two cyclic quotient singularities of type $\frac{1}{2}(1,1)$. Then, the singular fibers on $S$ are the union of a central component $Y$ and either none or exactly two mutually disjoint rational curves (the exceptional divisor $E$) which correspond to the resolution of the two singularities.\\
	In the first case we have that $Y^2=0$. In the second case, we use the formula given at the end of the section $2.2$ and we obtain that $Y^2=-1$. On the other hand, we have that 
	\[
		2g(Y)-2=K_Y.Y=(K_S+Y).Y=K_S.Y+Y^2
	\]
	where $K_S.Y\geq 0$ since $K_S$ is nef. Therefore, we obtain in both cases $g(Y)\geq 1$, which means that $Y$ is not rational.	
\end{proof}

\begin{Prop}\label{rational}
	Let $f:\mathbb{P}^1\rightarrow S$ be a non constant holomorphic map. Then 
	\[
		f(\mathbb{P}^1)\subset E\cup \mathbb{B}(K_S-E)
	\]
	where $\mathbb{B}(K_S-E):=\displaystyle\bigcap_{m>0}\Bs(m(K_S-E))$ is the stable locus base of $K_S-E$.
\end{Prop}

\begin{proof}
	Let $f:\mathbb{P}^1\rightarrow S$ be a non constant holomorphic map. Then, either $f$ is tangent to one of the foliations $\mathcal{F}_1$, $\mathcal{F}_2$ given by the fibrations $\sigma_1$, $\sigma_2$ respectively, or $f$ is not.
	
	If $f$ is tangent to one of the foliations, let us say to $\mathcal{F}_1$, then $f(\mathbb{P}^1)$ must be contained in a singular fiber of $\sigma_1:S\rightarrow C_1/G$. Otherwise $f(\mathbb{P}^1)$ would be contained in a smooth fiber, but smooth fibers are hyperbolic, since they are isomorphic to $C_2$ and $g(C_2)\geq 2$, and this is contradiction. Therefore, 
	\[
		f(\mathbb{P}^1)\subset Y\cup E
	\]
	where $Y$ is the central component  and $E$ is the exceptional divisor;
	however, lemma \ref{cc} tells us that $g(Y)\geq 1$ and hence $f(\mathbb{P}^1)\subset E$.
	
	Now, let us suppose that $f$ is not tangent to any of the foliations and let us consider the composition $\widehat{f}:=\varphi\circ f$. So we have the following diagram:
	
	\begin{center}
		\raisebox{-0.5\height}{\includegraphics[scale=1.25]{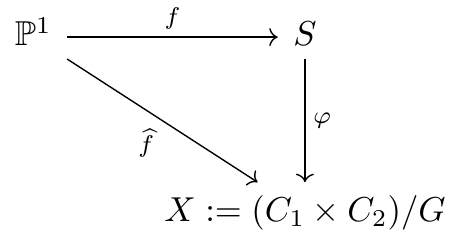}}
	\end{center}
	
	By Proposition \ref{2} we can consider a non zero section $\omega\in H^0(S,\mathcal{O}(m(K_S-E)))$ for $m$ large enough, and recall the section $\Theta(\omega)\in H^0(X^{reg},\mathcal{S}^{2m}\Omega_X)$ obtained via $\Theta$.  Then the section $\widehat{f}^*\Theta(\omega)=f^*\Gamma(\omega)$ vanishes because $H^0(\mathbb{P}^1,\Omega_{\mathbb{P}^1})=0$; moreover, since $\Theta(\omega)$ is locally written as $a(z_1,z_2){dz_1}^{m}{dz_2}^{m}$ and  $\widehat{f}$ is locally given by $\widehat{f}=(\widehat{f_1},\widehat{f_2})=\varphi(f_1,f_2)$ where $f_1,f_2$, $\widehat{f_1},\widehat{f_2}$ are holomorphic functions, the section $\widehat{f}^*\Theta(\omega)$ is locally given as
	\[
		a(\widehat{f_1},\widehat{f_2})(\widehat{f_1}')^m(\widehat{f_2}')^m=0
	\]	
	Thus we obtain that $a(\widehat{f_1},\widehat{f_2})=(a\circ\varphi)(f_1,f_2)=0$ since by hypothesis the other factors are not always equal to zero. This last equation means that the section $\omega$ vanishes on $f(\mathbb{P}^1)$, but this is true for any section of $\mathcal{O}(m(K_S-E))$, then $f(\mathbb{P}^1)\subset \Bs(m(K_S-E))$ and therefore $f(\mathbb{P}^1)\subset \mathbb{B}(K_S-E)$.
\end{proof}
	
	\subsection{Entire curves}
	We have already seen that the central components of singular fibers on $S$ are not rational, but we do not know yet if they can be elliptic. In the following example we will see that, in fact, for any product-quotient surface, the central components that do not intersect with the exceptional divisor, have genus bigger than one; however, in the case where the central components do intersect with the exceptional divisor, we give an example of a surface with a central component that is elliptic.

\begin{Ex}\label{ex}
	
	Let $S$ be a quotient-product surface and let us consider the natural fibration $\sigma_1:S\rightarrow C_1/G$ and a point $\overline{x}\in C_1/G$ with non trivial stabilizer $H_x$. Recall that the fiber $F$ of $\sigma_1$ over $\overline{x}$ is the union of a central component $Y\simeq C_1/H_x$ and either none or at least two mutually disjoint H-J strings which are in one-to-one correspondence with the branch points of $C_2\rightarrow C_2/H_x$. 
	
	In the first case, using the Riemann-Hurwitz formula, we obtain
    \[
	    2g(C_1)-2=|H_x|(2g(Y)-2),
	\]
    but $2g(C_1)-2>0$, then $g(Y)\geq 2$.
    
    For the second case, we suppose $S$ belongs to the first family of the table given in section $3$. Since $X=(C_1\times C_2)/G$ has only two singularities of type $\frac{1}{2}(1,1)$, then $C_2\rightarrow C_2/H_x$ has two branch points with multiplicity equal to $2$. Thus, using the Riemann-Hurwitz formula, we get
    \[
	    2g(C_1)-2=|H_x|(2g(Y)-1),
	\]
    but $g(C_1)=3$, then,
    \[
	    4=|H_x|(2g(Y)-1).
	\]
    We easily conclude that $|H_x|$ must be equal to $4$ and hence $g(Y)=1$.
\end{Ex}
    
Now, we recall a well known theorem asserting  that entire curves satisfy an algebraic differential equation. Namely:
   
\begin{Th}[(\cite{D}, Corollary 7.9), \cite{gr}]
    If there exists a non zero section $s\in H^0(S,\mathcal{S}^{m}\Omega_S\otimes A^{-1})$ with $A$ an ample line bundle and $m$ an integer, then for every entire curve $f:\mathbb{C}\rightarrow S$, $f^*s=0$.  
\end{Th}

Using this result we can follow the same argument used in Proposition \ref{rational}, to prove an analogous result for entire curves.  

\begin{Prop}\label{weakentire}
	Let $f:\mathbb{C}\rightarrow S$ be a non constant holomorphic map. Then 
	\[
		f(\mathbb{C})\subset \mathbf{Y}\cup E\cup \mathbb{B^+}(K_S-E)
	\]
	where $\mathbf{Y}$ is the union of all central components, $E$ is the exceptional divisor and 
	\[
		\mathbb{B^+}(K_S-E):=\bigcap_{m>0}\Bs(m(K_S-E)-A)
	\] 
	with $A$ an ample line bundle, is the augmented base locus of $K_S-E$.
\end{Prop}

\begin{proof}
	For any entire curve $f:\mathbb{C}\rightarrow S$ we also have the following two possibilities: either $f$ is tangent to one of the foliations $\mathcal{F}_1$, $\mathcal{F}_2$, or $f$ is not.
	In the first case we have again that $f(\mathbb{C})$ must be contained in the singular fibers because the smooth ones are Brody hyperbolic. Thus, $f(\mathbb{C})\subset \mathbf{Y}\cup E$.
	In the second case, the bigness of the line bundle $K_S-E$ ensures the existence of a non zero section $s\in H^0(S,\mathcal{O}(m(K_s-E))\otimes A^{-1})$ with $A$ ample and $m$ large enough, and via $\Gamma$ we obtain a non zero section $\Gamma(s)\in H^0(S,\mathcal{S}^{2m}\Omega_S\otimes A^{-1})$. So we have that $f^*\Gamma(s)=0$ and then, following the same argument than in the case of rational curves we obtain $f(\mathbb{C})\subset \mathbb{B}^+(m(K_s-E))$.
\end{proof}	

Note that Example \ref{ex} shows that, a priori, we can not avoid the central components because they could be elliptic, but a later result will show us that elliptic curves are contained in the the augmented base locus of $K_S-E$.
	
	\section{Effective version of Lang's conjecture}
	The purpose of this section is to prove Theorem \ref{effective}. Let us begin by recalling some basic facts that will be used. Let $S$ be a product-quotient surface, $C$ a smooth projective curve and $f:C\rightarrow S$ a holomorphic map such that $f(C)\nsubseteq E$. The differential map $df:T_C(-\log f^{-1}E)\rightarrow T_S(-\log E)$ induces a lifting $f_{[1]}: C\to \mathbb{P}(T_S(-\log E))$. Thus we get the following diagram:

\begin{center}
	\raisebox{-0.5\height}{\includegraphics[scale=1.25]{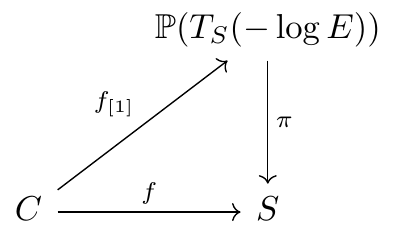}}
\end{center}

 Moreover, we have that  $\pi_*\mathcal{O}_{\mathbb{P}(T_S(-\log E))}(1)\simeq \Omega_S(\log E)$.
 On the other hand, recall that for each foliation $\mathcal{F}$ on $S$, we have the logarithmic exact sequence
 
\begin{center}
	\raisebox{-0.5\height}{\includegraphics[scale=1.25]{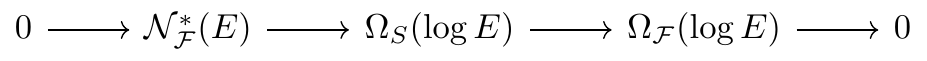}}
\end{center}

 and we can also define the divisor $Z:=\mathbb{P}(T_{\mathcal{F}}(-\log E))$ on $\mathbb{P}(T_{\mathcal{S}}(-\log E))$.

\begin{Lemma}\label{inequality}
	Let $\mathcal{F}$ be a foliation on $S$, $C$ a smooth projective surface and $f:C\rightarrow S$ a holomorphic map such that $f(C)\nsubseteq E$. If $f$ is not tangent to $\mathcal{F}$, then
	\[
		\deg f^*\mathcal{N}^*_{\mathcal{F}}(E)\leq 2g(C)-2+N_1(E)
	\]
	where $N_1(E)$ is the number of points on $f^{-1}(E)$ counted without multiplicities.
\end{Lemma}

\begin{proof}
	For the sake of simplicity we denote by $\mathcal{O}(1)$ the line bundle $\mathcal{O}_{\mathbb{P}(T_S(-\log E))}(1)$. Let us consider the exact sequence
	
	\begin{center}
		\raisebox{-0.5\height}{\includegraphics[scale=1.25]{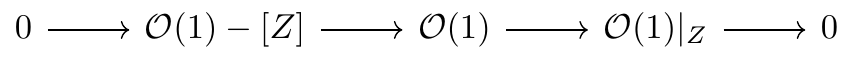}}
	\end{center}

	Now, taking the push-forwards we get 

	\begin{center}
		\raisebox{-0.5\height}{\includegraphics[scale=1.25]{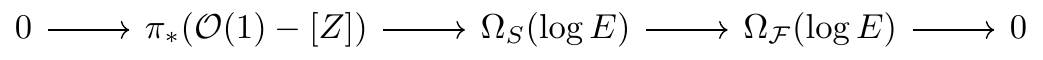}}
	\end{center}
	and thus we obtain $\pi_*(\mathcal{O}(1)-[Z])\simeq \mathcal{N}^*_{\mathcal{F}}(E)$. On the other hand, since $f$ is not tangent to $\mathcal{F}$ then $f_{[1]}(C)\nsubseteq Z$, thus $f_{[1]}(C).Z\geq 0$ and hence $\deg f^*_{[1]}[Z]\geq 0$. Therefore, 
	\[
		\deg f^*\mathcal{N}^*_{\mathcal{F}}(E)\leq \deg f^*_{[1]}\mathcal{O}(1).
	\]
	
	Moreover, the differential map 
	\[
		df:T_C(-f^{-1}(E))\longrightarrow f^*_{[1]}\mathcal{O}(-1)
	\]
	defines a non zero section of the line bundle $f^*_{[1]}\mathcal{O}(-1)\otimes K_C(f^{-1}(E))$ implying that this line bundle is effective. Then,
	\[
		\deg f^*_{[1]}\mathcal{O}(1)\leq \deg K_C(f^{-1}(E))=2g(C)-2+N_1(E).
	\]	
\end{proof}

\begin{Prop}\label{degree}
	Let $S$ be a product-quotient surface. If $f:C\rightarrow S$ is a holomorphic map such that $f(C)\nsubseteq E$, with $C$ a smooth projective curve and $E$ the exceptional divisor on $S$, then
	\[
		\deg f^*(K_S-E)\leq 2(2g(C)-2)
	\]
\end{Prop}

\begin{proof}
	First, let us suppose that $f$ is not tangent to any of the foliations $\mathcal{F}_1$, $\mathcal{F}_2$. Then by Lemma $\ref{inequality}$ we have that for $i=1,2$
	\[
		\deg f^*\mathcal{N}^*_{\mathcal{F}_i}(E)\leq 2g(C)-2+N_1(E).
	\]	
	Using Serrano's formula for the canonical bundle, we get that	
	\begin{align*}
		\deg f^*(K_S+E)
		&=\sum_{i=1}^2\deg f^*\mathcal{N}^*_{\mathcal{F}_i}(E)\\
		&\leq2(2g(C)-2+N_1(E))\\
		&= 2(2g(C)-2)+2N_1(E)
	\end{align*}	
	Therefore,
	\[
		\deg f^*(K_S-E)\leq 2(2g(C)-2).
	\]
	
	Now, we suppose that $f$ is tangent to one of the foliations, let us say to $\mathcal{F}_1$, then $f(C)$ is contained in a fiber $F$ of $\sigma_1:S\rightarrow C_1/G$. If $F$ is a smooth fiber we know that it is isomorphic to the curve $C_2$, and then,
	\[
		\deg f^*(K_S-E)=(K_S-E).C_2=K_{C_2}.C_2\leq K_C.C =2g(C)-2.
	\]
	
	If $F$ is a singular fiber, $f(C)$ must be contained in the central component $Y$ of the reduced structure of $F$ and hence,
	\[
		\deg f^*(K_S-E)=(K_S-E).Y=K_{Y}.Y-(Y.E+Y^2).
	\]
	When $F$ does not contain any H-J string we have $Y^2=0$ and $Y.E=0$; thus,
	\[
		\deg f^*(K_S-E)=K_{Y}.Y\leq K_C.C =2g(C)-2.
	\]
	On the other hand, when $F$ contains exactly $r$ H-J strings, $L_1,\cdots,L_r$, where each $L_i$ is the resolution of a cyclic quotient singularity of type $\frac{1}{n_i}(1,a_i)$, we have that
	\[
		Y.E+Y^2=r-\displaystyle\sum_{i=1}^r\dfrac{a_i}{n_i}\geq 0,
	\]
	and thus,
	\[
		\deg f^*(K_S-E)\leq K_{Y}.Y\leq K_C.C =2g(C)-2.
	\]
\end{proof}

As a consequence of the previous result we get first an alternative proof of Proposition \ref{rational}.

\begin{Cor}
	Let $S$ be a product-quotient surface of general type such that $P_g=0$ and $c_1^2=6$. If $f:\mathbb{P}^1\rightarrow S$ is a holomorphic map such that $f(\mathbb{P}^1)\nsubseteq E$, then
	\[
		f(\mathbb{P}^1)\subset \mathbb{B}(K_S-E)
	\]
	where $\mathbb{B}(K_S-E)$ is the stable base locus of $K_S-E$.
\end{Cor}

\begin{proof}
	We already know that in this case $K_S-E$ is big, so if $f(\mathbb{P}^1)\nsubseteq \mathbb{B}(K_S-E)$, $\deg f^*(K_S-E)\geq 0$. However, from Proposition $\ref{degree}$ we obtain $\deg f^*(K_S-E)\leq -4$. A contradiction. Thus $f(\mathbb{P}^1)\subset \mathbb{B}(K_S-E)$.
\end{proof}

In fact, we can also localize elliptic curves as shown by the following proposition.

\begin{Prop}\label{elliptic}
	If $S$ is a product-quotient surface of general type such that $P_g=0$ and $c_1^2=6$, and $f:C\rightarrow S$ is a holomorphic map where $C$ is a smooth projective curve of genus $g(C)=1$, then
	\[
		f(C)\subset \mathbb{B^+}(K_S-E)
	\]
	where $\mathbb{B^+}(K_S-E)$ is the augmented base locus of $K_S-E$.
\end{Prop}

\begin{proof}
	Since $K_S-E$ is big, then it can be written as the sum of an ample divisor $A$ and an effective divisor $D$. Moreover, the augmented base locus can be given in terms of all these possible sums as 
	\[
		\mathbb{B^+}(K_S-E)=\bigcap_{\mathclap{K_S-E=A+D}}\Supp D
	\]
	(\cite{base}, Remark $1.3$). Now, if $f(C)\nsubseteq \mathbb{B^+}(K_S-E)$ then there is a $D$ such that $f(C)\nsubseteq D$, thus $\deg f^*D\geq  0$ and hence,
	\[
		\deg f^*(K_S-E)=\deg f^*A+\deg f^*D>0,
	\]
	but note that $f(C)\nsubseteq E$, thus from Proposition \ref{degree} we have that
	\[
		\deg f^*(K_S-E)\leq 0.
	\]
	A contradiction. Therefore $f(C)\subset \mathbb{B^+}(K_S-E)$.
	
\end{proof}

Finally, we finish the proof of Theorem \ref{entire curves}.

\begin{Cor}
	If $S$ is a product-quotient surface of general type such that $P_g=0$ and $c_1^2=6$, then for any non constant holomorphic map $f:\mathbb{C}\rightarrow S$,
	\[
		f(\mathbb{C})\subset E\cup \mathbb{B^+}(K_S-E)
	\]
	where $E$ is the exceptional divisor and
	\[
		\mathbb{B^+}(K_S-E):=\bigcap_{m>0}\Bs(m(K_S-E)-A)
	\] 
	with $A$ an ample line bundle, is the augmented base locus of $K_S-E$.
\end{Cor}

\begin{proof}
	From Proposition \ref{weakentire} we have that $f(\mathbb{C})\subset \mathbf{Y} \cup E\cup \mathbb{B^+}(K_S-E)$ where $\mathbf{Y}$ is the union of all central components. First, note that we can remove all components with genus bigger or equal to 2 since they are hyperbolic, and by Lemma \ref{cc} we know that no component can be rational. Now, by Proposition \ref{elliptic}, the elliptic components must be contained in the augmented base locus of $K_S-E$. Therefore, $f(\mathbb{C})\subset E\cup \mathbb{B^+}(K_S-E)$.
\end{proof}

	\nocite{*}
	\bibliographystyle{alpha}
	\bibliography{bibliography}
	
\end{document}